\documentclass[10.5pt]{amsart}

\newtheorem{theoreme}{Theorem}[section]

\newtheorem{lemme}[theoreme]{Lemma}

\newtheorem{proposition}[theoreme]{Proposition}

\newtheorem{question}[theoreme]{Question}

\newtheorem{remarque}[theoreme]{Remark}

\newtheorem{exemple}[theoreme]{Example}

\input xy

\xyoption{all}

\usepackage{amscd}
\usepackage{amssymb}
\usepackage{enumitem}
\usepackage{leftidx}

\usepackage[a4paper, left=2cm,right=2cm,bottom=2cm,top=3.5cm]{geometry}
\usepackage{amsthm,array,amssymb,amscd,amsfonts,latexsym, url}
\usepackage{amsmath}
\usepackage{graphicx,epsfig}
\title{Basic remarks on Lagrangian submanifolds of hyperk\"ahler manifolds}
\author{Ren\'e Mboro}
\date{}
\begin{document}
\begin{abstract} This note presents basic restrictions on the topology ``general'' Lagrangian surfaces of hyper-K\"ahler $4$-folds and a remark on the interaction of a Lagrangian subvariety with a Lagrangian fibration of the associated hyper-K\"ahler variety.  
\end{abstract}
\maketitle
\section{Introduction}
Lagrangian submanifolds of irreducible symplectic manifolds are known to enjoy intersting properties. To name some of them, they are known to be projective (\cite[Proposition 2.1]{campana}) even when the symplectic manifold containing them is not and their deformations are unobstructed i.e. the corresponding Hilbert scheme is smooth at any point representing a smooth Lagrangian subvariety (see \cite[Section VI.6]{lehn}, see also \cite{voisin}). In this note, we present some additional properties of Lagrangian submanifolds.\\
\indent Among the most common examples of Lagrangian subvarieties of hyper-K\"ahler manifolds, we find curves on K3 surfaces and abelian varieties that appear, for example, as fibers of a Lagrangian fibration.\\
\indent In the first part of the note, we show that some of the features of these two model Lagrangian subvarieties are common to ``most'' of the other Lagrangian surfaces.\\
\indent In the case of curves on a $K3$, except for $\mathbb P^1$, the topological Euler characteristic is non-positive. The first result suggests that, essentially, the topological Euler characteristic of ``most'' Lagrangian subvarieties is of a given sign (determined by the dimension).  
\begin{proposition}\label{prop_euler_char} Let $S\subset Y$ be a Lagrangian surface in a hyper-K\"ahler $4$-fold whose deformations cover a dense open subset of $Y$. Then either \begin{itemize} 
\item two general surfaces parametrized by the same Hilbert scheme component as $S$, have empty intersection, in which case $\chi_{top}(S)=0$;
\item or there is a (possibly reducible) curve which is contained in every surface parametrized by the same Hilbert scheme component as $S$;
\item or $\chi_{top}(S)>0$
\end{itemize}
\end{proposition}

Although Lagrangian surfaces whose deformations cover the associated hyper-K\"ahler $4$-fold can have intermediate Kodaira dimension (i.e. $0$ or $1$), we have the following result about the Albanese dimension.

\begin{theoreme}\label{thm_alb_max} Let $S\subset Y$ be a Lagrangian surface in a projective hyper-K\"ahler $4$-fold whose deformations cover a dense open subset of $Y$. Then $S$ has maximal Albanese dimension.
\end{theoreme}
The result presented in the second part is concerned with the interaction of a Lagrangian subvariety $X\subset Y$ with a Lagrangian fibration of $Y$.

\begin{proposition}\label{prop_img_lagr_fibr} Let $\pi\colon Y\rightarrow \mathbb P^n$ be a Lagrangian fibration of a hyper-k\"ahler variety $(Y,\omega)$ endowed with a section.\\
\indent Let $X\subset Y$ be a smooth Lagrangian subvariety which is generically contained in the smooth locus of $\pi$ and is not a fiber. Then either $\pi_{|X}$ is generically finite or the general fiber of $\bar{\pi}_{|X}:X\rightarrow \pi(X)$ is a union of abelian varieties.
\end{proposition}

\section{Topology of Lagrangian surfaces}
Let $S\subset Y$ be a smooth Lagrangian surface of a hyper-K\"ahler $4$-fold. We have the classical isomorphism $\Omega_S\simeq N_{S/Y}$ which suggests that there is a deep interplay between the deformation theory of $S$ inside $Y$ and its topology -as already illustrated by the elementary remark (using that the deformations of $S$ are unobstructed) that if $S$ does not deform in $Y$ then $\pi_1(S)$ is finite-.\\
\indent Let us denote by $\mathcal H(Y)$ a dense open subset of the Hilbert scheme component containing $[S]$, parametrizing only smooth surfaces. We have the following:

\begin{lemme}\label{lem_dim_transverse} Assume that for general $[S],[S']\in \mathcal H(Y)$, $S\cap S'$ is a curve. Let us consider, for $[S]\in \mathcal H(Y)$, the map $e_{S}:\mathcal H(Y)\backslash \{[S]\}\rightarrow {\rm Div}^{\gamma_{S}}(S)$, $[S']\mapsto [S'\cap S]$ to the space of effective divisors whose class in $NS(S)$ is $\gamma_{S}=[S\cap S']$. Then for any $[S]\in \mathcal H(Y)$ the image of $e_S$ is a point. The latter is associated to the (possibly reducible) curve contained in every member of $\mathcal H(Y)$. 
\end{lemme}
\begin{proof} Note first that if $e_S$ has finite ($0$-dimensional) image for a surface $S$ the same is true for any other $[S']\in \mathcal H(Y)$.\\ 
\indent Indeed, as $\mathcal H(Y)$ is irreducible, $e_S(\mathcal H(Y)\backslash \{[S]\})$ consists of a point $[C]\in Div^{\gamma_S}(S)$. Take another $[S']\in \mathcal H(Y)$. By definition, $C\subset S'$ and $[C]=\gamma_{S'}$ in $NS(S')$. For any other $[S'']\in \mathcal H(Y)$, as $C= S''\cap S$, $C\subset S''\cap S'$. Since $[S''\cap S']=\gamma_{S'}=[C]$ in $NS(S')$, we get $C=S''\cap S'$ i.e. ${\rm Im}(e_{S'})=\{[C]\}$.\\
\indent So we just have to prove that there is no surface for which $e_S$ has positive dimensional image.\\

\indent So assume there is a surface $S$ for which ${\rm Im}(e_S)$ is positive dimensional. Then by what we have just seen, ${\rm Im}(e_{S'})$ is positive dimensional for any $[S']\in\mathcal H(Y)$. Taking general hyperplanes sections of $\mathcal H(Y)$, we can find a closed (irreducible) subvariety  $M\subset \mathcal H(Y)$ such that $f:\mathcal S_M\rightarrow Y$, where $\mathcal S_M$ is the pullback of the universal surface on $M$, is generically finite (dominant) and the restriction $e_{S|M}$ of $e_{S}$ to $M$ has positive dimensional image for any $[S]\in M$. In particular ${\rm dim}(M)=2$.\\

\indent If ${\rm dim}({\rm Im}(e_{S|M}))=2$ for a surface $[S]\in M$. Then the general (non-empty) fiber of $e_{S|M}$ is $0$-dimensional. For a general $[S']\in M$, denoting $C=e_{S|M}(S')$, we have $C=S\cap S'\in {\rm Im}(e_{S'|M})\subset Div^{\gamma_{S'}}(S')$. Any other $[S'']\in e_{S'|M}^{-1}([C])$ is also in $e_{S|M}^{-1}([C])$ (as $C\subset S''\cap S$ and $[C]=[S''\cap S]$ in $NS(S)$) so there are finitely of them i.e. $e_{S'|M}^{-1}([C])$ is $0$-dimensional. So %by semi-continuity of the dimension of the fiber, 
the general fiber of $e_{S'|M}$ is also $0$-dimensional, in other words ${\rm dim}({\rm Im}(e_{S'}))=2$ for the general $[S']\in M$.\\
\indent Then the curves $S'\cap S$ cover a dense open subset of $S$ and as the deformations of $S$ parametrized by $M$ cover a dense open subset of $Y$, those curves cover a dense open subset of $Y$. Denoting $\mathcal C$ the universal curve over ${\rm Div}^{\gamma_{S}}(S)$, the generic fiber of $g_{S}:e_{S|M}^*\mathcal C\rightarrow S$ has dimension $1$. Now take a general point $y\in Y$ and a general surface $[S]\in M$ passing through $y$. Then there is a $1$-dimensional family of divisors of $S$ of the form $S'\cap S$, with $[S']\in M$ passing through $y$. So $f$ is not generically finite.\\

\indent We are left with the case when ${\rm dim}({\rm Im}(e_{S|M}))=1$ for the general surface $[S]\in M$. In this case, the curves $S'\cap S$ cover again $S$. As the deformations of $S$ parametrized by $M$ cover $Y$, those curves cover $Y$. Let $y\in Y$ be a general point and $C_y$ a curve of the form $S_1\cap S_2$ containing $y$. As the general fiber of $e_{S_1|M}$ is $1$-dimensional, there is a $1$-dimensional family $[S_t]\in M$ such that $S_1\cap S_t=C_y\ni y$. So $f$ is not generically finite.
\end{proof}

\begin{proof}[Proof of Proposition \ref{prop_euler_char}]
For two surfaces $[S_1],[S_2]$ in $\mathcal H(Y)$, we have $$\begin{aligned} \int_Y[S_1]\cdot[S_2]&=\int_Y[S]\cdot[S]=\int_Si^*i_*(1)= \int_Sc_{top}(N_{S/Y})=\chi_{top}(S).\end{aligned}$$ using $N_{S/Y}\simeq \Omega_S$ and where $i:S\hookrightarrow Y$.\\
\indent According to Lemma \ref{lem_dim_transverse}, if there is no common curve to the surfaces in $\mathcal H(Y)$, the intersection of two general surfaces is either empty, in which case $0=[S_1]\cdot [S_2]=\chi_{top}(S)$ or $0$-dimensional $0<[S_1]\cdot [S_2]=\chi_{top}(S)$.
\end{proof}

\begin{remarque} {\rm (1) As the following example shows, it is necessary that the deformations of $S$ cover $Y$\rm : Let $C\subset \Sigma$ be a smooth curve of genus $>1$ on a $K3$ surface. Then, as explained in \cite[Section 8]{igb_master_thesis}, $\mathbb P(T_{\Sigma|C})\subset \Sigma^{[2]}$ is Lagrangian and} $$\begin{aligned}\chi_{top}(\mathbb P(T_{\Sigma|C}))&= \int c_2(T_{\mathbb P(T_{\Sigma|C})})\\ &=\int (pr^*c_1(T_{\Sigma|C})+2c_1(\mathcal O_{\mathbb P(T_{\Sigma|C})}(1)))\cdot pr^*c_1(T_C)\\ &=2(2-2g(C)) <0.\end{aligned}$$ {\rm As $h^0(\Omega_{\mathbb P(T_{\Sigma|C})})=h^0(\omega_C)$ and deformations of the Lagrangian subvarieties $C\subset \Sigma$ and $\mathbb P(T_{\Sigma|C})\subset \Sigma^{[2]}$ are unobstructed, the deformations of $\mathbb P(T_{\Sigma|C})$ are induced by deformations of $C$; so the deformations of $\mathbb P(T_{\Sigma|C})$ are contained in the branched locus $E\simeq \mathbb P(T_S)$ of $q:Bl_{\Delta_{\Sigma^2}}(\Sigma^2)\rightarrow \Sigma^{[2]}$.}\\
\indent {\rm Moreover $\mathbb P(T_{\Sigma|C_1})\cap \mathbb P(T_{\Sigma|C_2})=\cup_{i=1}^{2g(C)-2}\mathbb P(T_{\Sigma,p_i})$ where $\{p_1,\dots,p_{2g(C)-2}\}=C_1\cap C_2$ and as $\mathcal O_\Sigma(C)$ is base point free (\cite[Lemma 2.3]{Huybrechts_k3}), the surfaces $\mathbb P(T_{\Sigma|C})$ have no curve in common.}\\

\indent {\rm Actually, in this case, any Lagrangian surface $S\subset E$ is of the form $\mathbb P(T_{\Sigma|C})$ for a smooth curve $C\subset \Sigma$.}\\
\indent {\rm Indeed, denoting by $\tau: Bl_{\Delta_{\Sigma^2}}(\Sigma^2)\rightarrow \Sigma^2$ the blow-up, by $\omega$ the symplectic form on $\Sigma^{[2]}$ and by $\omega_\Sigma$ the one on $\Sigma$, we have $q^*\omega =\tau^*(pr_1^*\omega_\Sigma+pr_2^*\omega_\Sigma)$. So, denoting $j_E:E\hookrightarrow Bl_{\Delta_{\Sigma^2}}(\Sigma^2)$, we get} $$\begin{aligned}j_E^*q^*\omega = j_E^*\tau^*(pr_1^*\omega_\Sigma+pr_2^*\omega_\Sigma) =\tau_{|E}^*i_{\Delta_{\Sigma^2}}(pr_1^*\omega_\Sigma+pr_2^*\omega_\Sigma) =2\tau_{|E}^*\omega_\Sigma.\end{aligned}$$ 
\indent {\rm Now, let $S\subset E$ be a Lagrangian surface. Assume $\tau_{|S}:S\rightarrow \Sigma$ is generically finite. Let $U\subset \Sigma$ be open subset over which $\tau_{|S}$ is \'etale. By the above description of the restriction of the symplectic form, $\omega_{|S}$ is non-degenerate (symplectic) on $\tau_{|S}^{-1}(U)$, in particular non-zero, thus $S$ is not Lagrangian. So $\tau_{|S}$ cannot be surjective and by semi-continuity of the dimension of the fibers of $\bar{\tau}_{|S}:S\rightarrow {\rm Im}(\tau_{|S})$, all the fibers are $1$-dimensional i.e. are fibers of $\tau_{|E}:E\rightarrow\Sigma$. Denoting $C\subset \Sigma$ its image, we get $S\simeq \mathbb P(T_{\Sigma|C})$. So $C$ is smooth.}\\
\indent {\rm In particular, any deformation of a Lagrangian surface $S$ contained in the rigid uniruled divisor $E$ stays in the latter.}\\

\indent {\rm (2) Examples of families covering $Y$ whose members have curves in common can be constructed. Let $C\subset \Sigma$ be a smooth curve of genus $\geq 2$ on a $K3$ surface and $V_p\subset |\mathcal O_{\Sigma}(C)|$ be the linear  system of curves passing through $p\in \Sigma$ (generically with multiplicity one). The curves parametrized by $V_p$ cover $\Sigma$.}\\
\indent {\rm For a general pair $([C_1],[C_2])\in V_p^2/i$, $i$ being the natural involution, $C_1$ and $C_2$ intersect transversally (in particular at $p$). As explained in \cite[Section 8]{igb_master_thesis}, $Bl_{C_1\cap C_2}(C_1\times C_2)\subset \Sigma^{[2]}$ is a Lagrangian submanifold.}\\
\indent {\rm All the smooth Lagrangian submanifolds parametrized by an open subset of  $V_p^2/i$ contain the rational curve $\mathbb P(T_{\Sigma,p})$. Moreover the universal surface over this open subset dominates $\Sigma^{[2]}$ since one can find a member of $V_p$ through any point of $\Sigma$.}\\
\indent {\rm As $\mathcal O_\Sigma(C)$ is base point free (\cite[Lemma 2.3]{Huybrechts_k3}), for the full Hilbert scheme of $Bl_{C_1\cap C_2}(C_1\times C_2)\subset \Sigma^{[2]}$, there is no curve common to every member of it. Moreover, the example does not contradict $\chi_{top}(S)\geq 0$.}
\end{remarque}

The following strenghtening of Lemma \ref{lem_dim_transverse} is natural:
\begin{question}\label{question_hilbert_sch} (General position for Hilbert scheme): Given a Hilbert scheme component $H$ parametrizing (generically) smooth (irreducible) subvarieties of a smooth projective variety $Y$ that cover it, is the intersection of two general members of $H$ dimensionally transverse?
\end{question}

\indent Now, let $S$ be a projective surface of Albanese dimension $1$. Then the Albanese morphism factors through a fibration (flat with connected fibers) $\widetilde{alb_S}:S\rightarrow B$ over a smooth curve of genus $q(S)=h^{1,0}(S)$ (\cite{beauv_cplx_surf}) such that $J(B)\simeq Alb(S)$.\\
\indent We recall that having Albanese dimension $1$ is a topological property (\cite[Section 2.2]{lopes_pardini}).\\ 
\indent Now assume $S\subset Y$ is a Lagrangian surface of a projective hyper-K\"ahler that has Albanese dimension $1$ and such that the deformations of $S$ in $Y$ cover the latter. Then we have $2\leq {\rm dim}(\mathcal H(Y))=h^{1,0}(S)=g(B)$ (using that the deformations of $S$ are unobstructed and $\Omega_S\simeq N_{S/Y}$) and any surface parametrized by $\mathcal H(Y)$ has then Albanese dimension $1$.\\
\indent Up to taking an \'etale cover of (an open subset of) $\mathcal H(Y)$ (which does not change the tangent spaces) we have %$$\xymatrix{\mathcal S\ar[r]^{alb_{\mathcal S}}\ar[d] &\mathcal Alb(\mathcal S)\ar[ld]\\ \mathcal H(Y) &}$$ and get 
the relative Albanese fibration $$\xymatrix{\mathcal S\ar[r]^{\widetilde{alb_{\mathcal S}}}\ar[d] &\mathcal B\ar[ld]\\ \mathcal H(Y). &}$$

\indent We have $H^0(N_{S/Y})\simeq H^0(\Omega_S)\overset{\widetilde{alb_S}^*}{\simeq} H^0(\omega_B)$ so that the general section $\sigma_S=\widetilde{alb_S}^*\sigma_B$ of $\Omega_S$ vanishes along a disjoint union $Z_{\sigma_S}=\cup_{i=1}^{2g(B)-2}F_i=[\widetilde{alb_S}^{-1}(div(\sigma_B))]$ of smooth fibers of $\widetilde{alb_S}$.\\
\indent One can think of the zero locus of a section of $N_{S/Y}$ as an ``infinitesimal intersection'' of $S$ with one of its first order deformation. So here this infinitesimal intersection is not dimensionally transverse. Theorem \ref{thm_alb_max} will follow from Lemma \ref{lem_dim_transverse} after we have proven that this infinitesimal picture can be integrated to a dimensionally non-transverse actual intersection.\\
\indent More precisely, we will prove that the deformations of $S$ fixing $Z=Z_{\sigma_S}$ are unobstructed.\\
\indent Let us first prove that the fact that such $\mathcal O_S(Z)$ comes from the base of the fibration is an open property.
\begin{lemme}\label{lem_from_base} Let $S'\subset Y$ be a deformation of $S$ fixing $Z$. Then $\mathcal O_{S'}(Z)$ can be written $\widetilde{alb_{S'}}^*\mathcal L$ for a degree $2g(B')-2$ line bundle 
\end{lemme}
\begin{proof}
As $\widetilde{alb_S}^*\omega_B\simeq \mathcal O_S(Z)$ and having trivial Chern class is topological, $[\widetilde{alb_{S'}}^*\kappa_{B'}]=[Z]$ in $NS(S')$. So we can write $\widetilde{alb_{S'}}^*\kappa_{B'} = Z + \ell$ in ${\rm Pic}(S')$ for some $\ell\in {\rm Pic}^0(S')$. Intersecting with $\widetilde{alb_{S'}}^*\kappa_{B'}$ we get $0=\int_{S'}\widetilde{alb_{S'}}^*\kappa_{B'}^2=\int_{S'}Z\cdot \widetilde{alb_{S'}}^*\kappa_{B'}+\int_{S'}\ell\cdot \widetilde{alb_{S'}}^*\kappa_{B'}$. As $\ell\in {\rm Pic}^0(S')$, we get $\int_{S'}Z\cdot \widetilde{alb_{S'}}^*\kappa_{B'}=0$. But as $g(B')\geq 2$, $\omega_{B'}$ is ample, so $Z$ is contracted by $\widetilde{alb_{S'}}$. Now let $H$ be a ample divisor on $Y$, the invariance of $H$-degree of the fibers of the Albanese fibrations gives that the components of $Z$ are actual fibers of $\widetilde{alb_{S'}}$ (not just components of fibers). Hence the result ($\widetilde{alb_{S'}}$ is flat).
\end{proof}
As any degree $2g-2$ line bundle $L$ of a genus $g$ curve $B$ is either $\omega_B$ or satisfies $h^0(L)=g-1$,  the difference between $\mathcal O_{S'}(Z)$ and $\widetilde{alb_{S'}}^*\omega_B$ can be detected by their respective number of linearly independent sections. So we will prove that all the sections of $\mathcal O_S(Z)\simeq \widetilde{alb_S}^*\omega_B$ deform.\\

\indent Let us recall some results of deformation theory. They can be found in \cite{lehn} and \cite{sernesi} for example.\\
\indent Set the deformation functor 
$$\begin{tabular}{llll}
$H_S^{Z,Y}:$ &$\mathcal A$ &$\rightarrow$ &$Sets$\\
$ $ &$A$ &$\mapsto$ &$\{Z\times {\rm Spec}(A)\subset \mathcal S_A\subset Y\times {\rm Spec}(A),\ \mathcal S_A\ {\rm flat\ over}\ A\ {\rm and}\ \mathcal S_A\otimes k\simeq S\}$
\end{tabular}$$
where $\mathcal A$ is the category of local artinian $\mathbb C$-algebra with resiude field $\mathbb C$.\\
\indent By cocycle computations, we get the following proposition (which is essentially \cite[Lemma I.4.3]{lehn}).

\begin{proposition}\label{prop_def_funct_representable} The functor $H_S^{Z,Y}$ is pro representable (subfunctor of the local Hilbert functor) with tangent space $H^0(N_{S/Y}(-Z))$ and obstruction space $H^1(N_{S/Y}(-Z))$.
\end{proposition}

We will prove that $H_S^{Z,Y}$ is unobstructed i.e. that any infinitesimal deformation can be extended to an actual deformation. Letting $A_n=\mathbb C[t]/(t^{n+1})$, we have the following criterion:
\begin{proposition}\label{prop_unobstruct_criterion}(\cite[Corollary I.1.7]{lehn}) The functor $H_S^{Z,Y}$ is unobstructed if and only if for any $n\geq 0$ the natural map $H_S^{Z,Y}(A_{n+1})\rightarrow H_S^{Z,Y}(A_n)$ (induced by $A_{n+1}\twoheadrightarrow A_n$) is surjective.
\end{proposition}

We will use the $T^1$-lifting principle (see \cite[section VI.3.6]{lehn}) to prove that the maps $H_S^{Z,Y}(A_{n+1})\rightarrow H_S^{Z,Y}(A_n)$ are surjective.\\
\indent Let us introduce $D_n=A_n[\epsilon]/(\epsilon^2)$, $C_n=D_n/(\epsilon t^n)$. There are projections $C_{n+1}\rightarrow D_n$, $D_n\rightarrow C_n\rightarrow A_n$. Let us also introduce the homomorphism of $\mathbb C$-algebras $\delta:A_{n+1}\rightarrow D_n$, $t\mapsto t+\epsilon$. It is injective. Likewise, let us define $\delta':A_n\rightarrow C_n$, $t\mapsto t+\epsilon$.\\
\indent Given a $[S_n]\in H_S^{Z,Y}(A_n)$, we denote $H_Z^{Z,Y}(D_n)_{S_n}$ the fiber of $H_S^{Z,Y}(D_n)\rightarrow H_S^{Z,Y}(A_n)$ over $[S_n]$.\\
\indent The $T^1$-lifting principle consists of the following
\begin{proposition}\label{prop_statement_t_1_lifting}(\cite[Lemmas VI.3.7, VI.3.8]{lehn}) For a given $n\geq 0$, if for any $[S_{n+1}]\in H_S^{Z,Y}(A_{n+1})$, the map $H_S^{Z,Y}(D_{n+1})_{S_{n+1}}\rightarrow H_S^{Z,Y}(D_n)_{S_n}$, where $S_n=S_{n+1|A_n}$, is surjective then $H_S^{Z,Y}(D_{n+1})\rightarrow H_S^{Z,Y}(C_{n+1})$ and $H_S^{Z,Y}(A_{n+2})\rightarrow H_S^{Z,Y}(A_{n+1})$ are surjective.
\end{proposition}
\indent Actually the surjectivity of $H_S^{Z,Y}(A_{n+2})\rightarrow H_S^{Z,Y}(A_{n+1})$ is derived from the surjectivity of $H_S^{Z,Y}(D_{n+1})\rightarrow H_S^{Z,Y}(C_{n+1})$ by applying $H_S^{Z,Y}$ to the commutative diagram (\cite[Section VI.3.6]{lehn}) $$\xymatrix{0\ar[r] &\mathbb C\cdot [t^{n+2}]\ar[d]^{(n+2)\cdot}_{\cong}\ar[r] &A_{n+2}\ar[d]^{\delta}\ar[r] &A_{n+1}\ar[r]\ar[d]^{\delta'} &0\\ 0\ar[r] &\mathbb C\cdot [\epsilon t^{n+1}]\ar[r] &D_{n+1}\ar[r] &C_{n+1}\ar[r] &0.}$$

\indent We recall the following results of deformation theory that can be found in \cite{sernesi} for first order deformation. Lifting objects from $A_n$ to $D_n$ works a lot like first order deformations of the objects.

The following Lemma is essentially \cite[Lemma I.4.3]{lehn}.

\begin{lemme}\label{lem_lift_Hilb_A_n_to_B_n} For $[S_n]\in H_S^{Z,Y}(A_n)$, there is a natural $1$ to $1$ correspondence between $H_S^{Z,Y}(D_n)_{S_n}$ and $H^0(N_{S_n/Y\times {\rm Spec}(A_n)}(-(Z\times {\rm Spec}(A_n))))$. 
\end{lemme}

\indent Following \cite{sernesi}, we introduce tools to deal with deformation of lines bundles. The proof are almost identical to those found in \textit{loc. cit.} one has just to be careful of some additional automorphisms that appear.\\
\indent Since $S$ is smooth, any infinitesimal deformation of $S$ is locally trivial (\cite[Theorem 1.2.4]{sernesi}).\\
\indent Let $\mathcal V=\{V_i\}$ be an affine open cover $S$ and $[S_n]\in H_S^{Z,Y}(A_n)$. We have $A_n$-isomorphisms $\theta_i:V_i\times {\rm Spec}(A_n)\rightarrow S_{n|V_i}$ which gives for each (ordered) pair $(i,j)$ the gluing automorphisms $\theta_{ij}=\theta_i^{-1}\theta_j:V_{ij}\times {\rm Spec}(A_n)\rightarrow V_{ij}\times {\rm Spec}(A_n)$. The proof of the following proposition follows the one of \cite[Proposition 1.2.9]{sernesi}.
\begin{proposition}\label{prop_lift_abstract_A_n_to_B_n} For $[S_n]\in H_S^{Z,Y}(A_n)$, there is a $1$ to $1$ correspondence between $$\{S_{D_n}\rightarrow {\rm Spec}(D_n),\ {\rm flat\ and\ restricting\ to}\ S_n\}/isom$$ and $H^1(T_{S_n/A_n})$.
\end{proposition}
\begin{proof} We just show one direction. Let $S_{D_n}\rightarrow {\rm Spec}(D_n)$ be an extension of $S_n\rightarrow {\rm Spec}(A_n)$. We can choose the affine cover $\mathcal V$ of $S$ so that $S_{D_n|V_i}$ and $S_{n|V_i}$ are trivial. So we have isomorphisms of deformations $$\theta_i:V_i\times {\rm Spec}(A_n)\rightarrow S_{n|V_i}\ {\rm and}\ \widetilde{\theta_i}:V_i\times {\rm Spec}(D_n)\rightarrow S_{D_n|V_i}$$ with $\widetilde{\theta_i}$ restricting to $\theta_i$ and automorphisms $$\theta_{ij}:V_{ij}\times {\rm Spec}(A_n)\rightarrow V_{ij}\times {\rm Spec}(A_n)\ {\rm and}\ \widetilde{\theta_{ij}}:V_{ij}\times {\rm Spec}(D_n)\rightarrow V_{ij}\times {\rm Spec}(D_n).$$\\
\indent Set $V_{ij}\simeq {\rm Spec}(B_{ij})$. Denoting $i_n:B_{ij}\otimes A_n\rightarrow B_{ij}\otimes D_n$ the natural inclusion, $q_n:B_{ij}\otimes D_n\rightarrow B_{ij}\otimes A_n$ the surjective $B_{ij}\otimes A_n$-algebra homomorphism and $p_n:B_{ij}\otimes D_n\rightarrow B_{ij}\otimes A_n$ the projection on the $\epsilon$ component (which is a $B_{ij}\otimes A_n$-homomorphism), we can define $\eta_{ij}=i_n\circ \theta_{ij}\circ q_n+\epsilon i_n\circ \theta_{ij}\circ p_n$. A direct calculation shows that it is an endomorphism of $D_n$-algebras.\\
\indent By assumption $\widetilde{\theta_{ij}}-\eta_{ij}$ is trivial ${\rm mod}\ \epsilon$. Writing $\epsilon \phi_{ij}=\widetilde{\theta_{ij}}-\eta_{ij}$, for $x,y\in B_{ij}\otimes D_n$, we find $\phi_{ij}(xy)=\theta_{ij}(q_n(x))\phi_{ij}(y)+\theta_{ij}(q_n(y))\phi_{ij}(x)$, which can be written $$\begin{aligned}(\theta_{ij}^{-1}\phi_{ij})(xy)= q_n(x)(\theta_{ij}^{-1}\phi_{ij})(y)+q_n(y)(\theta_{ij}^{-1}\phi_{ij})(x)=x\cdot (\theta_{ij}^{-1}\phi_{ij})(y)+ y\cdot (\theta_{ij}^{-1}\phi_{ij})(x)\end{aligned}$$ i.e. $(\theta_{ij}^{-1}\phi_{ij})\in Der_{B_{ij}\otimes D_n}(B_{ij}\otimes D_n,B_{ij}\otimes A_n)\simeq Hom_{B_{ij}\otimes D_n}(\Omega_{B_{ij}\otimes D_n/D_n}, B_{ij}\otimes A_n)\simeq Der_{B_{ij}\otimes A_n}(B_{ij}\otimes A_n,B_{ij}\otimes A_n)=\Gamma(V_{ij}\times {\rm Spec}(A_n),T_{V_{ij}\times {\rm Spec}(A_n)/{\rm Spec}(A_n)})$.\\

\indent Since on $V_{ijk}\times {\rm Spec}(D_n)$, $\widetilde{\theta_{ij}}\widetilde{\theta_{jk}}=\widetilde{\theta_{ik}}$, looking at $\phi_{ij}$ as a map from $B_{ij}\otimes A_n$ to itself, we find $$\theta_{ij}\phi_{jk}+\phi_{ij}\theta_{jk}=\phi_{ik}$$ which can be re-written $$\underbrace{\theta_{ij}\theta_{jk}}_{\theta_{ik}}(\theta^{-1}_{jk}\phi_{jk})+\theta_{ij}(\theta_{ij}^{-1}\phi_{ij})\theta_{jk}=\theta_{ik}(\theta_{ik}^{-1}\phi_{ik})$$ and pre-composing with $\theta_k^{-1}$ and post-composing with $\theta_i$, we get $$\theta_j(\theta_{ij}^{-1}\phi_{ij})\theta_j^{-1}+ \theta_k(\theta_{jk}^{-1}\phi_{jk})\theta_k^{-1}=\theta_k(\theta_{ik}^{-1}\phi_{ik})\theta_k^{-1}$$ and means that $\{\theta_j(\theta_{ij}^{-1}\phi_{ij})\theta_j^{-1}=\theta_i\phi_{ij}\theta_j^{-1}\}$ defines a \v Cech $1$-cocycle i.e. an element of $H^1(T_{S_n/A_n})$.
\end{proof}

\indent Let $d:\mathcal O_{S_n}\rightarrow \Omega_{S_n/A_n}$ be the natural $A_n$-derivation. We can define a homomorphism of sheaves of abelian groups $\mathcal O_{S_n}^*\rightarrow \Omega_{S_n/A_n}$, $a\mapsto \frac{da}{a}$. It gives rise to a group homomorphism $c:H^1(S_n,\mathcal O_{S_n}^*)\rightarrow H^1(S_n,\Omega_{S_n/A_n})$.\\
\indent As $\Omega_{S_n/A_n}$ is locally free, $H^1(S_n, \Omega_{S_n/A_n})\simeq {\rm Ext}^1(T_{S_n/A_n},\mathcal O_{S_n})$ so that to any line bundle $L_n\in Pic(S_n)$ we can associate an extension $$0\rightarrow \mathcal O_{S_n}\rightarrow \mathcal E_{L_n}\rightarrow T_{S_n/A_n}\rightarrow 0$$ defined by $c(L_n)$.\\
\indent For a line bundle $L_n$ represented by the cocycle $\{(V_{ij}\times {\rm Spec}(A_n),f_{ij})\}$ with $f_{ij}\in \Gamma(V_{ij}\times {\rm Spec}(A_n),\mathcal O_{V_{ij}\times {\rm Spec}(A_n)}^*)$, the sheaf $\mathcal E_{L_n|\theta_i(V_{i}\times {\rm Spec}(A_n))}$ is isomorphic to $(\mathcal O_{S_n}\oplus T_{S_n/A_n})_{|\theta_i(V_{i}\times {\rm Spec}(A_n))}$ and sections $(a_i,d_i)$ of $(\mathcal O_{S_n}\oplus T_{S_n/A_n})_{|\theta_i(V_{i}\times {\rm Spec}(A_n))}$ and $(a_j,d_j)$ of $(\mathcal O_{S_n}\oplus T_{S_n/A_n})_{|\theta_j(V_{j}\times {\rm Spec}(A_n))}$ are identified on $\theta_i(V_{i}\times {\rm Spec}(A_n))\cap \theta_j(V_{j}\times {\rm Spec}(A_n))$ if and only if $d_i=d_j$ and $a_j-a_i=\frac{d_i(\theta_i(f_{ij}))}{\theta_i(f_{ij})}$\\
\indent(we recall that as the cocyle relation translates into $\theta_i(f_{ij})\theta_j(f_{jk})=\theta_i(f_{ik})$ for any triple, so that, for example, $f_{ji}\theta_{ji}(f_{ij})=f_{jj}=1$).\\ 

The proof of the following theorem follows the one of \cite[Theorem 3.3.11]{sernesi}.

\begin{theoreme}\label{thm_lift_pairs_A_n_to_B_n} Let $(S_n,L_n)$ be a (projective) deformation of a pair $(S,L)$ (with $S$ smooth projective) over $A_n$. There is a $1$ to $1$ correspondence between $$\{(S_{D_n},L_{D_n}),\ {\rm lifting}\ (S_n,L_n)\ {\rm on}\ D_n\}/isom$$ and $H^1(S_n,\mathcal E_{L_n})$.\\

\indent If $L_{D_n}$ has cocycle representation $\{(V_{ij,n})\times {\rm Spec}(A_1), f_{ij}+\epsilon g_{ij})\}$, (using $D_n\simeq A_n\otimes A_1$) where $V_{ij,n}\subset S_n$ is the affine open subset isomorphic to $V_{ij}\times {\rm Spec}(A_n)$, restricting to $V_{ij}\subset S$, with $g_{ij}\in \Gamma(V_{ij,n}\times {\rm Spec}(A_1),\mathcal O_{V_{ij,n}\times {\rm Spec}(A_1)})$, and $\{(V_{ij,n}, d_{ij})\}\in H^1(T_{S_n/A_n})$ is the class of the extension $S_{D_n}$ of $S_n$, then the associated class in $H^1(\mathcal E_{L_n})$ is represented by the cocyle $\{(V_{ij,n}, (\frac{g_{ij}}{f_{ij}},d_{ij})\}$. 
\end{theoreme}

\indent Given a deformation $(S_n,L_n)$ of a pair $(S,L)$ over $A_n$, we can define a homomorphism of sheaves $$M:\mathcal E_{L_n}\rightarrow H^0(S_n,L_n)^\vee\otimes_{A_n} L_n$$ in the following way: let $\{(V_{ij}\times {\rm Spec}(A_n), f_{ij})\}$ be a cocycle representation of $L_n$. Let $V\subset S_n$ be an open set and $\eta\in \Gamma(V,\mathcal E_{L_n})$; it is given by a system $\{(a_i,d_i)\in \Gamma(V\cap \theta_i(V_{i}\times {\rm Spec}(A_n)),\mathcal O_{S_n})\times \Gamma(V\cap \theta_i(V_{i}\times {\rm Spec}(A_n)),T_{S_n/A_n})_i\}$ such that $d_i=d_j$ and $a_j-a_i=\frac{ d_i(\theta_i(f_{ij})}{\theta_i(f_{ij})}$ on $V\cap \theta_i(V_{ij}\times {\rm Spec}(A_n))$. For every section $s=\{(s_i\in \Gamma(\theta_i(V_{i}\times {\rm Spec}(A_n),\mathcal O_{S_n})_i\}\in H^0(S_n,L_n)$ set $$M(\eta)(s_i)=a_is_i+d_i(s_i).$$ 
As done in \cite[3.3.4]{sernesi}, a direct calculation, that on $V\cap \theta_i(V_{ij}\times {\rm Spec}(A_n))$, $f_{ij}\theta_j^{-1}(M(\eta)(s_j))=\theta_i^{-1}(M(\eta)(s_i))$ so that $M(\eta)(s)\in \Gamma(V,L_n)$.\\
\indent Let $\eta\in H^1(\mathcal E_{L_n})$ be given by the system $\{\theta_j(V_{ij,n}),(a_{ij},d_{ij}))\}$; $M$ induces a $A_n$-linear map $$\begin{tabular}{llll}
$M_1(\eta):$ &$H^0(L_n)$ &$\rightarrow$ &$H^1(L_n)$\\
$ $ &$(s_i)$ &$\to$ &$\overline{(a_{ij}s_i+d_{ij}(s_i))}.$
\end{tabular}$$
The proof of the following proposition follows the one of \cite[Proposition 3.3.14]{sernesi}. 
\begin{proposition}\label{prop_lift_section_A_n_to_B_n} Let $(S_n,L_n)$ be a deformation of the pair $(S,L)$ over $A_n$ and $(S_{D_n},L_{D_n})$ a lifting of $(S_n,L_n)$ to $D_n$ defined by a class $\eta\in H^1(\mathcal E_{L_n})$. Then a section $s\in H^0(L_n)$ extends to a section of $L_{D_n}$ if and only if $s\in {\rm ker}(M_1(\eta))$.
\end{proposition}

\indent Now, we can prove Theorem \ref{thm_alb_max}.\\

\begin{proof}[Proof of Theorem \ref{thm_alb_max}] Let us prove that $H_S^{Z,Y}$ is unobstruted by induction. Let $\mathcal U=\{U_i\simeq {\rm Spec}(P_i))\}$ be an affine open cover of $Y$ and for each $i$, $(x_{1,i},x_{2,i},x_{3,i})$ be a regular sequence such that $S\cap U_i\simeq {\rm Spec}(P_i/(x_{1,i},x_{2,i}))$ and $Z\cap U_i\simeq {\rm Spec}(P_i/(x_{1,i},x_{2,i},x_{3,i}))$.\\
\indent The line bundle $\mathcal O_S(Z)$ has the following cocycle representation $\{(U_{ij}\cap S,\frac{\overline{x_{3,i}}}{\overline{x_{3,j}}})\}$.\\
\indent We have $h^0(\Omega_S(-Z))=h^0(\Omega_S(\widetilde{alb_S}^*T_B))=1$; so let $\sigma\in H^0(\Omega_S(-Z))\simeq H^0(N_{S/Y}(-Z))\simeq Hom_{\mathcal O_S}(\mathcal I_{S/Y}/\mathcal I_{S/Y}^2,\mathcal I_{Z/Y}/\mathcal I_{S/Y})$ be a generator. It gives rise to a first order deformation $S_1\subset Y\times {\rm Spec}(A_1)$ of $S$ fixing $Z$. Looking at $\sigma_{U_i\cap S}\in Hom_{S\cap U_i}(P_i/(x_{1,i},x_{2,i})[x_{1,i}]\oplus P_i/(x_{1,i},x_{2,i})[x_{2,i}], (\overline{x_{3,i}}))$ as a couple acting by scalar product we have $\sigma_{U_i\cap S}=(x_{3,i}a_i,x_{3,i}b_i)$ for some $a_i,b_i\in P_i/(x_{1,i},x_{2,i})$ and $S_1\cap (U_i\times {\rm Spec}(A_1)\simeq {\rm Spec}(P_i\otimes A_1/(x_{1,i}+tx_{3,i}a_i,x_{2,i}+tx_{3,i}b_i))$, with $t^2=0$.\\
\indent As $(x_{1,i},x_{2,i},x_{3,i})/(x_{1,i}+tx_{3,i}a_i,x_{2,i}+tx_{3,i}b_i)\simeq (\overline{x_{3,i}})$, the Cartier divisor $Z\times {\rm Spec}(A_1)$ on $S_1$ admits the representation $\{(S_1\cap (U_i\times {\rm Spec}(A_1)), \overline{x_{3,i}})\}$. We also have the associated line bundle $\mathcal O_{S_1}(Z\times {\rm Spec}(A_1))$ on $S_1$.\\

\indent Looking at the \v Cech resolution of the exact sequence $$0\rightarrow T_S\rightarrow T_{Y|S}\rightarrow N_{S/Y}\rightarrow 0$$ by the snake lemma, we see that the class $\kappa_0(\sigma)\in H^1(T_S)$ associated to this first order deformation is represnted by a cocycle $\{(U_{ij}\cap S, \tilde d_{i|U_{ij}}-\tilde d_{j|U_{ij}})\}$, where $d_i\in Der_{\mathbb C}(P_i,P_i/(x_{1,i},x_{2,i}))$ satisfies $d_i(x_{1,i})=x_{3,i}a_i$, $d_i(x_{2,i})=x_{3,i}b_i$ and $\tilde d_i$ is the associated abstract (before pre and post composition by coordinate chart). Given those derivations, we can write an explicit trivialization of the open subset $S_1\cap (U_i\times {\rm Spec}(A_1))$ suited to Theorem \ref{thm_lift_pairs_A_n_to_B_n}: consider the commutative diagram with exact rows
$$\xymatrix{0\ar[r] &(x_{1,i},x_{2,i})\otimes A_1\ar[r]\ar[d]^{\psi_i} &P_i\otimes A_1\ar[d]^{\psi_i=1+td_i}\ar[r] &(P_i/(x_{1,i},x_{2,i}))\otimes A_1\ar[r]\ar[d]^{\overline\psi_i} &0\\
0\ar[r] &(\psi_i(x_{1,i}),\psi_i(x_{2,i})\ar[r] &P_i\otimes A_1\ar[r] &P_i\otimes A_1/(x_{1,i}+tx_{3,i}a_i,x_{2,i}+tx_{3,i}b_i)\ar[r] &0}$$
where $\overline\psi_i$ is the isomorphism induced by the automorphism (as $d_i$ is a derivation) $\psi_i$. Then under this trivialization the local equation of $Z\times {\rm Spec}(A_1)$ is $x_{3,i}-td_i(x_{3,i})$ so that the interesting cocycle representation (in view of Theorem \ref{thm_lift_pairs_A_n_to_B_n}) of $\mathcal O_{S_1}(Z\times {\rm Spec}(A_1))$ is $\{((S\cap U_{ij})\times {\rm Spec}(A_1), \frac{x_{3,i}}{x_{3,j}}(1+t(\frac{d_j(x_{3,j})}{x_{3,j}}-\frac{d_i(x_{3,i})}{x_{3,i}})))\}$.\\

\indent So the class $\eta_0(\sigma)\in H^1(\mathcal E_{\mathcal O_S(Z)})$ associated to this first order deformation of the pair $(S,\mathcal O_S(Z))$ induced by $\sigma\in H^0(N_{S/Y})$ has, according to Theorem \ref{thm_lift_pairs_A_n_to_B_n}, cocycle representation $\{(U_{ij}\cap S,(\frac{d_j(x_{3,j})}{x_{3,j}}-\frac{d_i(x_{3,i})}{x_{3,i}},\tilde d_{i|U_{ij}}-\tilde d_{j|U_{ij}}))\}$. Now, let $s\in H^0(\mathcal O_S(Z))\simeq H^0(\widetilde{alb_S}^*\omega_B)$ represented by $\{(U_i\cap S,s_i)\}$. We have $$\begin{aligned} M_1(\eta_0(\sigma))(s)&=\{(U_{ij}\cap S, \overline{\frac{d_j(x_{3,j})}{x_{3,j}}s_i-\frac{d_i(x_{3,i})}{x_{3,i}}s_i+d_i(s_i)-(\frac{x_{3,i}}{x_{3,j}}\cdot )\circ d_j\circ(\frac{x_{3,j}}{x_{3,i}}\cdot) (s_i)}\}\\ &=\{(U_{ij}\cap S, d_i(s_i)-\frac{d_i(x_{3,i})}{x_{3,i}}s_i - \frac{x_{3,i}}{x_{3,j}}(d_j(s_j) -\frac{d_j(x_{3,j})}{x_{3,j}}s_j)\}\\ &=\delta(\{(U_i\cap S,d_i(s_i)-\frac{d_i(x_{3,i})}{x_{3,i}}s_i)\})\end{aligned}$$ where $\delta$ is the \v Cech differential i.e. $M_1(\eta_0(\sigma))=0\in H^1(\mathcal O_S(Z))$. So by Proposition \ref{prop_lift_section_A_n_to_B_n}, any section $s\in H^0(\mathcal O_S(Z))$ extends to a section of $\mathcal O_{S_1}(Z\times {\rm Spec}(A_1))$. 
%caution: might need to look at $(\frac{d_i(x_{3,i})}{x_{3,i}}s_i+d_i(s_i))\in \Gamma(U_i\cap S, L)$ and $\frac{d_i(x_{3,i})}{x_{3,i}}s_i+d_i(s_i)-(\frac{d_j(x_{3,j})}{x_{3,j}}s_j+d_j(s_j))=d_i(s_i)-d_j(s_j)$,
% $M_1(\eta_0(\sigma))(s)$ is a \v Cech coboundary, so trivial in $H^1(\mathcal O_S(Z))$. Thus any section of $\mathcal O_S(Z)$ extends to a section of $\mathcal O_{S_1}(Z\times {\rm Spec}(A_1))$. 
In particular $h^0(\mathcal O_{S_1}(Z\times {\rm Spec}(A_1)))=h^{1,0}(S)$, which by the discussion after Lemma \ref{lem_from_base}, implies $\mathcal O_{S_1}(Z\times {\rm Spec}(A_1))\simeq \widetilde{alb_{S_1}}^*\omega_{B_1/A_1}$.\\

\indent Now we go through the $T^1$-lifting principle. We have $h^0(N_{S_1/Y\times {\rm Spec}(A_1)}(-(Z\times {\rm Spec}(A_1))=h^0(\Omega_{S_1/A_1}(\widetilde{alb_{S_1/A_1}}^*T_{B_1/A_1})=1$. Choose a generator $\sigma_1\in H^0(N_{S_1/Y\times {\rm Spec}(A_1)}(-(Z\times {\rm Spec}(A_1)))$. By Lemma \ref{lem_lift_Hilb_A_n_to_B_n}, it gives rise to an extension $S_{D_1}$ of $S_1$ to $D_1$ fixing $Z$. We have $\sigma_{|U_i\cap S\times {\rm Spec}(A_1)}=(\overline{x_{3,i}}a'_i,\overline{x_{3,i}}b'_i)$ for some $a'_i,b'_i\in P_i\otimes A_1/(x_{1,i}+tx_{3,i}a_i,x_{2,i}+tx_{3,i}b_i)$ and then $S_{D_1}\cap (U_i\times {\rm Spec}(D_1))\simeq {\rm Spec}(P_i\otimes B_1/(x_{1,i}+tx_{3,i}a_i+\epsilon x_{3,i}a'_i,x_{2,i}+tx_{3,i}b_i+\epsilon x_{3,i}b'_i))$.\\
\indent Using the exact sequence $$0\rightarrow T_{S_1/A_1}\rightarrow T_{Y\times {\rm Spec}(A_1)/{\rm Spec}(A_1)|S_1}\rightarrow N_{S_1/Y\times {\rm Spec}(A_1)}\rightarrow 0$$ we see that the associated class $\kappa_1(\sigma_1)\in H^1(T_{S_1/A_1})$ is represented by $\{(S_1\cap (U_{ij}\times {\rm Spec}(A_1)), \tilde d'_{i|U_{ij}}-\tilde d'_{j|U_{ij}})\}$, where $d'_i\in Der_{A_1}(P_i\otimes A_1,P_i\otimes A_1/(x_{1,i}+tx_{3,i}a_i,x_{2,i}+tx_{3,i}b_i))$ such that $d'_i(x_{1,i}+tx_{3,i}a_i)=x_{3,i}a'_i$, $d'_i(x_{2,i}+tx_{3,i}b_i)=x_{3,i}b'_i$ and $\tilde d_i'$ is the associated abstract (before pre and post composing with trivializations) derivation.\\
\indent As $(x_{1,i},x_{2,i},x_{3,i})/(x_{1,i}+tx_{3,i}a_i+\epsilon x_{3,i}a'_i,x_{2,i}+tx_{3,i}b_i+\epsilon x_{3,i}b'_i)\simeq (\overline{x_{3,i}})$, the Cartier divisor $Z\times {\rm Spec}(D_1)$ has representation $\{(S_{D_1}\cap (U_i\times {\rm Spec}(D_1), \overline{x_{3,i}})\}$.\\
\indent Recalling the natural isomorphism $A_n\otimes A_1\simeq D_n$, we have the commutative diagram with exact rows
$$\resizebox{40em}{!}{\xymatrix@C=1em{0\ar[r] &(x_{1,i}+tx_{3,i}a_i,x_{2,i}+tx_{3,i}b_i)\otimes A_1\ar[r]\ar[d]^{\psi_i'} &P_i\otimes D_1\ar[r]\ar[d]^{\psi_i'=1+\epsilon d_i'} &(P_i\otimes A_1/(x_{1,i}+tx_{3,i}a_i,x_{2,i}+tx_{3,i}b_i))\otimes A_1\ar[r]\ar[d]^{\overline{\psi_i'}} &0 \\
0 \ar[r] &(x_{1,i}+x_{3,i}(ta_i+\epsilon a_i'),x_{2,i}+x_{3,i}(tb_i+\epsilon b_i'))\ar[r] &P_i\otimes D_1\ar[r] &P_i\otimes D_1/(x_{1,i}+x_{3,i}(ta_i+\epsilon a_i'),x_{2,i}+x_{3,i}(tb_i+\epsilon b_i'))\ar[r] &0}}$$ 
where $\widetilde{\psi_i'}$ is the isomorphism induced by the automorphism $\psi_i'$. Under this trivialization, the local equation of $Z\times {\rm Spec}(D_1)$ is $x_{3,i}-\epsilon d_i'(x_{3,i})$. So the interesting cocycle representation of $\mathcal O_{S_{D_1}}(Z\times {\rm Spec}(D_1))$ is $\{((S_1\cap (U_{ij}\times {\rm Spec}(A_1))\times {\rm Spec}(A_1), \frac{x_{3,i}}{x_{3,j}}(1+\epsilon(\frac{d_j(x_{3,j})}{x_{3,j}}-\frac{d_i(x_{3,i})}{x_{3,i}})))\}$ and the class $\eta_1(\sigma_1)\in H^1(\mathcal E_{\mathcal O_{S_1}(Z\times {\rm Spec}(A_1)})$ associated to the extension of the pair $(S_1,\mathcal O_{S_1}(Z\times {\rm Spec}(A_1)))$ to $D_1$ has the following cocycle representation $\{((S_1\cap (U_{ij}\times {\rm Spec}(A_1)), (\frac{d_j(x_{3,j})}{x_{3,j}}-\frac{d_i(x_{3,i})}{x_{3,i}}, \tilde d_i'-\tilde d_j'))\}$.\\
\indent A similar computation as above shows that $M_1(\eta_1(\sigma_1))(s)=0\in H^1(\mathcal O_{S_1}(Z\times {\rm Spec}(A_1))$ for any section $s\in H^0(\mathcal O_{S_1}(Z\times {\rm Spec}(A_1))$. So by Proposition \ref{prop_lift_section_A_n_to_B_n}, $h^0(\mathcal O_{S_1}(Z\times {\rm Spec}(A_1))=h^0(\mathcal O_{S_{D_1}}(Z\times {\rm Spec}(D_1))$ i.e. $\mathcal O_{S_{D_1}}(Z\times {\rm Spec}(D_1)\simeq \widetilde{alb_{S_{D_1}}}^*\omega_{B_{D_1}/D_1}$ by the discussion after Lemma \ref{lem_from_base}.\\

\indent The map $H_S^{Z,Y}(D_1)_{S_1}\simeq H^0(\Omega_{S_1/A_1}(-(Z\times {\rm Spec}(A_1)))\rightarrow H_S^{Z,Y}(D_0)_S\simeq H^0(\Omega_S(-Z))$ is obviously surjective. So by Proposition \ref{prop_statement_t_1_lifting}, $H_S^{Z,Y}(A_2)\rightarrow H_S^{Z,Y}(A_1)$ is also surjective i.e. there is an extension $S_2$ of $S_1$ fixing $Z$.\\
\indent As mentioned after Proposition \ref{prop_statement_t_1_lifting}, we have a commutative diagram $$\xymatrix{S_2\in H_S^{Y,Z}(A_2)\ar[d]^{H_S^{Y,Z}(\delta)}\ar[r] &H_S^{Y,Z}(A_1)\ni S_1\ar[d]^{H_S^{Y,Z}(\delta')}\\S_{D_1}\in H_S^{Y,Z}(D_1)\ar[r] &H_S^{Y,Z}(C_1).}$$
\indent Now, as $\delta:A_2\rightarrow D_1$ is injective, the image of ${\rm Spec}(D_1)\rightarrow {\rm Spec}(A_2)$ is dense. Since the component of the (relative) Picard scheme of $S$ containing $\mathcal O_S(Z)$ is proper and $\mathcal O_{S_{D_1}}(Z\times {\rm Spec}(D_1))\simeq \widetilde{alb_{S_{D_1}}}^*\omega_{B_{D_1}}$, we get $\mathcal O_{S_2}(Z\times {\rm Spec}(A_2))\simeq \widetilde{alb_{S_2}}^*\omega_{B_2}$. So again $h^0(N_{S_2/Y\times {\rm Spec}(A_2)}(-(Z\times {\rm Spec}(A_2))))=1$ and we can continue the induction.\\
\indent So a first order deformation of $S$ fixing $Z$ extends to an actual deformation of $S$ fixing $Z$. So for a general $S'$ in the Hilbert scheme component of $S$, $S\cap S'$ meet along a divisor of the linear system $|\widetilde{alb_S}^*\omega_B|$. As $\widetilde{alb_B}$ is base point free, Lemma \ref{lem_dim_transverse} gives a contradiction. 
\end{proof}

\section{Interaction with Lagrangian fibrations}
We recall the following Proposition found in \cite[Proposition 3.5]{hwang_base}:
\begin{proposition}\label{prop_hwang} Let $\pi\colon M\rightarrow B$ be a smooth Lagrangian fibration with a Lagrangian section. There is a unramified surjective holomorphic map $f:T^*B\rightarrow M$, $T^*B$ being the total space of the cotangent bundle of $B$, which commutes with the projection to $B$ and the map $\pi$. Moreover the pull-back of the symplectic form $\omega$ on $M$ by $f$ coincides with the standard symplectic structure of $T^*B$
\end{proposition}

\begin{proof}[Proof of Proposition \ref{prop_img_lagr_fibr}] Let $U\subset \mathbb P^n$ be the maximal Zariski open subset over which $\pi$ is smooth. According to the above Proposition \ref{prop_hwang}, there is a unramified surjective holomorphic morphism $f:T^*U\rightarrow \pi^{-1}(U)$, which commutes with the respective projections on $U$ and such that $f^*\omega$ is the canonical symplectic form $\omega_{can}$ on $T^*U$.\\
\indent Assume $X\cap \pi^{-1}(U)\neq \emptyset$. It is then a Lagrangian submanifold of $(\pi^{-1}(U),\omega_{|\pi^{-1}(U)})$.\\
\indent As $f$ is a local biholomorphism, $Z^\circ:=f^{-1}(X\cap \pi^{-1}(U))$ is a manifold. The same reason implies that $Z^\circ$ is Lagrangian submanifold of $(T^*U,\omega_{can})$.\\
\indent Assume $\pi(Z^\circ)$ %(why should it be reduced??)
 is a proper subvariety of $U$. Choose a coordinates chart $(U',(z_1,\dots,z_n))$ of $U$ centered at a point $u\in pr(Z^\circ)=\pi(X)\cap U$ which is a smooth point of $pr(Z^\circ)$ and over which $\pi_{|X}$ (thus $pr$) is smooth, so that $pr(Z^\circ)\cap U'$ is defined by $pr(Z^\circ)\cap U'=\{z_{k+1}=\cdots=z_n=0\}$.\\
\indent We recall that in this chart, $\omega_{can}=\sum_id\xi_i\wedge dz_i$ where $(z_1,\dots,z_n,\xi_1,\dots,\xi_n)$ are the cotangent coordinates associated to $(z_1,\dots,z_n)$. Let us analyze the affine manifold $Z^\circ_u\subset p^{-1}(u)\simeq T_u^*U\simeq \mathbb C^n$.\\
\indent For any $z\in Z_u^\circ$, as $pr$ (because $\pi_{|X}$ is) smooth above $u$, we have the exact sequence $$0\rightarrow T_{(u,z),pr}\rightarrow T_{z}Z^\circ\rightarrow T_upr(Z^\circ)\rightarrow 0$$ and an isomorphism $T_zZ_u^\circ\simeq T_{(u,z),pr}$.\\
\indent The space $T_upr(Z^\circ)$ is generated by $\frac{\partial}{\partial z_1},\dots,\frac{\partial}{\partial z_k}$ and any $v\in T_zZ_u^\circ$ can be written $v=\sum_ia_i\frac{\partial}{\partial \xi_i}$. As $Z^\circ$ is Lagrangian, we have $$\begin{aligned}0=\omega_{can}(v,\frac{\partial}{\partial z_i})=a_i,\ i=1,\dots, r.\end{aligned}$$ So, we have $T_zZ_u^\circ\subset Span(\frac{\partial}{\partial \xi_i},k+1\leq i\leq n)$, which gives an equality as they have the same dimension.\\
\indent The tangent space of the complex manifold $Z_u^\circ\subset\mathbb C^n$ at each point is thus  equal (as an affine space) to a fixed subspace, so $Z_u^\circ$ is a finite union of linear spaces.\\
\indent So the generic fiber of $pr_{|Z^\circ}$ is linear. Now looking at the projection $f:Z^\circ\rightarrow X\cap \pi^{-1}(U)$ we see that the general fibers of $\pi_{|X}$ are (union of) compact complex manifolds that admit a surjective (unramified) holomorphic morphism from an affine space, so they are complex tori.

\end{proof}

\begin{exemple} {\rm (1) Let $\Sigma$ be an Enriques surface and $q:S\rightarrow \Sigma$ its universal cover. Take $L\in {\rm Pic}(\Sigma)$ a line bundle giving rise to an elliptic fibration on $\Sigma\rightarrow |L|$. Then $q^*L$ also gives rise to an elliptic fibration on $S$.}\\
\indent {\rm Now, $q^*L$ induces a Lagrangian fibration on $\pi:S^{[2]}\rightarrow \mathbb P^2$. On the other hand, we have an embedding $\Sigma\hookrightarrow S^{[2]}$, $x\mapsto q^{-1}(x)$ and as $h^{2,0}(\Sigma)=0$, $\Sigma$ is a Lagrangian subvariety of $S^{[2]}$. Then $\pi_{|\Sigma}$ is the elliptic fibration of $\Sigma$ given by $L$ and $\pi(\Sigma)$ is a conic.}\\
\indent {\rm (2) Let $C\subset S$ a smooth curve which is a multisection of a $K3$ surface admitting an elliptic fibration $S\rightarrow \mathbb P^1$. Then $C^{(2)}\subset S^{[2]}$ is a Lagrangian subvariety for which $f_{|C^{(2)}}$ is a finite morphism.} 
\end{exemple}
\begin{remarque} {\rm In these examples, $\bar{\pi}_{|X}:X\rightarrow \pi(X)$ is flat.}
\end{remarque}

\section*{Acknowledgments}
I would like to thank Daniel Huybrechts for referring me to Christian Lehn's thesis (in which I found key tools to complete this work) for a proof of unobstructedness of deformations of Lagrangian subvarieties and for suggestions on a preliminary version of this work. I would like to thank also Lie Fu for useful discussions around Question \ref{question_hilbert_sch}.\\
\indent I was partially supported by Simons Investigators Award HMS, HSE University Basic Research Program, National Science Fund of Bulgaria, National Scientific Program “Excellent Research and People for the Development of European Science” (VIHREN), Project No. KP-06-DV-7.

\noindent \begin{tabular}[t]{l}
\textit{rene.mboro@polytechnique.edu}\\
UMiami Miami, HSE Moscow,\\
Institute of Mathematics and Informatics, Bulgarian Academy of Sciences,\\
Acad. G. Bonchev Str. bl. 8, 1113, Sofia, Bulgaria.
\end{tabular}\\
\end{document}